\documentclass[11pt]{amsart}

\newcommand{\preprint}[1]{}

\newcommand{\hide}[1]{}

\usepackage{amssymb}
\usepackage{amsbsy}
\usepackage{amscd}
\usepackage{amsmath}
\usepackage{amsthm}
\usepackage{bbold}
\usepackage{mathrsfs}
\input xy
\xyoption{all}

\numberwithin{equation}{section}

\theoremstyle{plain}
\newtheorem{thm}{Theorem}[section]
\newtheorem{prop}[thm]{Proposition}

\newtheorem{cor}[thm]{Corollary}
\newtheorem{lem}[thm]{Lemma}

\theoremstyle{definition}
\newtheorem{defi}[thm]{Definition}

\theoremstyle{remark}

\newtheorem{rem}[thm]{Remark}

\newcommand{\C}{{\mathscr C}}

\newcommand{\K}{{\mathscr K}}

\newcommand{\fM}{{\mathfrak M}}

\newcommand{\N}{{\mathscr N}}
\newcommand{\NN}{{\mathbb N}}

\newcommand{\PP}{{\mathbb P}}

\newcommand{\W}{{\mathscr W}}

\newcommand{\X}{{\mathscr X}}

\newcommand{\RealNumbers}{{\mathbb R}}
\newcommand{\RR}{{\mathbb R}}
\newcommand{\Integers}{{\mathbb Z}}
\newcommand{\ZZ}{{\mathbb Z}}
\newcommand{\ComplexNumbers}{{\mathbb C}}
\newcommand{\RationalNumbers}{{\mathbb Q}}

\newcommand{\rank}{{\rm rank}}

\newcommand{\Aut}{{\rm Aut}}

\newcommand{\Wedge}[1]{\stackrel{#1}{\wedge}}

\begin{document}
\title[Hilbert schemes of $K3$ surfaces are dense in moduli]
{Hilbert schemes of $K3$ surfaces are dense in moduli}
\date{\today}
\author{E. Markman}
\address{Department of Mathematics and Statistics, 
University of Massachusetts, Amherst, MA 01003}
\email{markman@math.umass.edu}
\author{S. Mehrotra}
\address{Department of Mathematics, University of Wisconsin, Madison, WI 53706}
\email{mehrotra@math.umass.edu}
\date{\today}

\begin{abstract}
We prove that the locus of Hilbert schemes of $n$ points on a projective $K3$ surface
is dense in the moduli space of irreducible holomorphic symplectic manifolds of that deformation type.
The analogous result for generalized Kummer manifolds is proven as well.
\end{abstract}

\maketitle

\tableofcontents 
%
\section{Introduction}

An {\em irreducible holomorphic symplectic manifold} is a simply connected compact K\"{a}hler manifold $X$,
such that $H^0(X,\Wedge{2}T^*X)$ is one-dimensional spanned by an everywhere non-degenerate
holomorphic $2$-form. The second cohomology $H^2(X,\Integers)$ of such an $X$ admits 
a unique non-degenerate symmetric integral and indivisible bilinear pairing, called
the  {\em Beauville-Bogomolov pairing}, which is a topological invariant \cite{beauville}.

A {\em marking} for an irreducible holomorphic symplectic manifold
$X$ is a choice of an isometry
$\eta:H^2(X,\Integers)\rightarrow \Lambda$ with a fixed lattice $\Lambda$. Two marked pairs
$(X_1,\eta_1)$, $(X_2,\eta_2)$ are {\em isomorphic}, if there exists 
an isomorphism $f:X_1\rightarrow X_2$, such that 
$\eta_1\circ f^*=\eta_2$.
There exists a coarse moduli space $\fM_\Lambda$ parametrizing
isomorphism classes of marked pairs \cite{huybrechts-basic-results}.
$\fM_\Lambda$ is a smooth complex manifold of dimension $b_2(X)-2$,
but it is non-Hausdorff.

Let $X$ be a K\"{a}hler manifold, which is  deformation equivalent to
the Hilbert scheme $S^{[n]}$, of length $n$ subschemes of a $K3$ surface $S$, $n\geq 2$.
Then $X$ is an  irreducible holomorphic symplectic manifold and 
$H^2(X,\Integers)$, endowed with its Beauville-Bogomolov pairing,
is isometric to the lattice $\Lambda$ recalled below \cite{beauville}. 
Let $H$ be the unimodular lattice of rank $2$ and signature $(1,1)$. Given an integer $d$, denote by
$\langle d\rangle$ the rank $1$ lattice generated by an element $x$ satisfying $(x,x)=d$. 
Let $E_8(-1)$ be the negative definite lattice of type $E_8$. 
Set
\begin{eqnarray*}
\Lambda_{K3} & := & E_8(-1)\oplus E_8(-1) \oplus H \oplus H \oplus H,
\\
\Lambda & := &  \Lambda_{K3}\oplus \langle 2-2n\rangle.
\end{eqnarray*}
The direct sums above are orthogonal. 
The lattice $\Lambda_{K3}$  is unimodular, while $\Lambda$ is not.

Let $\fM_\Lambda^0$ be a connected component of $\fM_\Lambda$ containing a marked pair
$(S^{[n]},\eta)$,
where $S^{[n]}$ is the Hilbert scheme of length $n$ sub-schemes of a $K3$ surface $S$. 

\begin{thm}
\label{thm-density-in-moduli}
The locus in $\fM_\Lambda^0$, consisting of marked pairs $(X,\eta)$, where $X$ is isomorphic to 
the Hilbert scheme $S^{[n]}$, for some projective $K3$ surface $S$, is dense in $\fM_\Lambda^0$.
\end{thm}

The proof of the theorem is concluded in section \ref{sec-density-in-moduli}. 
In section \ref{sec-density-of-kummers} we state and prove the analogous Theorem \ref{thm-density-of-kummers}
for generalized Kummer varieties. We also compute the monodromy group of 
$2n$ dimensional generalized Kummer varieties, when $n+1$ is a prime power 
(Corollary \ref{cor-monodromy-of-generalized-kummers-when-n+1-is-a-prime-power}).
Section 2 contains lattice-theoretic lemmas preparatory to these density results.

Researchers in the field have understood for some time that the Torelli Theorem for irreducible holomorphic symplectic manifolds
should imply Theorem \ref{thm-density-in-moduli}. This Torelli Theorem was recently proven by Verbitsky \cite{verbitsky}, 
so the time is ripe to provide a write-up of the proof of Theorem \ref{thm-density-in-moduli}.
The result will be used in an essential way in our up-coming work  
\cite{MS}. A related density result was proven recently by Anan'in and Verbitsky
\cite{AV}. 

%
\section{Density of periods}
\label{sec-density-of-periods}

The {\em period} of a marked pair $(X,\eta)$ is the
line $\eta[H^{2,0}(X)]$ considered as a point in the projective space
$\PP[\Lambda\otimes_\Integers\ComplexNumbers]$. The period lies in the period domain 
\begin{equation}
\label{eq-period-domain}
\Omega \ := \ \{
p \ : \ (p,p)=0 \ \ \ \mbox{and} \ \ \ 
(p,\bar{p}) > 0
\}.
\end{equation}
$\Omega$ is an open subset, in the classical topology, of the quadric in 
$\PP[\Lambda\otimes\ComplexNumbers]$ of isotropic lines \cite{beauville}. 
The period map 
\begin{eqnarray}
\label{eq-period-map}
P \ : \ {\mathfrak M}_\Lambda & \longrightarrow & \Omega,
\\
\nonumber
(X,\eta) & \mapsto & \eta[H^{2,0}(X)]
\end{eqnarray}
is a local isomorphism, by the Local Torelli Theorem \cite{beauville}.

Given a point $p\in \Omega$, set 
$\Lambda^{1,1}(p):=\{\lambda\in \Lambda \ : \ (\lambda,p)=0\}$. Note that 
$\Lambda^{1,1}(p)$ is a sublattice of $\Lambda$ and 
$\Lambda^{1,1}(p)=(0)$, if $p$ does not belong to 
the countable union of hyperplane sections
$\cup_{\lambda\in \Lambda\setminus\{0\}}[\lambda^\perp\cap\Omega]$.
Given a marked pair $(X,\eta)$, we get the isomorphism
$H^{1,1}(X,\Integers)\cong\Lambda^{1,1}(P(X,\eta))$, via the restriction of $\eta$.

\hide{
\begin{defi}
\label{def-positive-cone}
Let $X$ be an irreducible holomorphic symplectic manifold.
The cone 
$
\{\alpha\in H^{1,1}(X,\RealNumbers) \ : \ (\alpha,\alpha)>0\}
$
has two connected components. 
The {\em positive cone} $\C_X$
is the connected component containing the K\"{a}hler cone 
$\K_X$.
\end{defi}


Two points $x$ and $y$ of a topological space $M$ are 
{\em inseparable}, if every pair of open subsets $U$, $V$, 
with $x\in U$ and $y\in V$, have a non-empty intersection $U\cap V$.
A point $x\in M$ is a {\em Hausdorff point}, 
if there does not exist any point $y\in [M\setminus\{x\}]$,
such that $x$ and $y$ are inseparable.

\begin{thm} 
\label{thm-global-torelli}
(The Global Torelli Theorem)
Fix a connected component $\fM^0_\Lambda$ of $\fM_\Lambda$.
\begin{enumerate}
\item
(\cite{huybrechts-basic-results}, Theorem 8.1)
The period map $P$ restricts to a surjective holomorphic map
$P_0:\fM^0_\Lambda\rightarrow \Omega$. 
\item 
\label{thm-item-injectivity}
(\cite{verbitsky}, Theorem 1.16)
The fiber $P_0^{-1}(p)$ consists of pairwise 
inseparable points, for all $p\in \Omega$.
\item
\label{thm-item-inseparable-are-bimeromorphic}
(\cite{huybrechts-basic-results}, Theorem 4.3)
Let  $(X_1,\eta_1)$ and $(X_2,\eta_2)$ be two inseparable points of $\fM_\Lambda$.
Then $X_1$ and $X_2$ are bimeromorphic.
\item
\label{thm-item-single-Hausdorff-point}
The marked pair $(X,\eta)$ is a Hausdorff point of $\fM_\Lambda$,
if and only if $\C_X=\K_X$.
\item
\label{thm-item-cyclic-picard-and-projective-imply-K-X=C-X}
The fiber $P_0^{-1}(p)$, $p\in \Omega$, consists of a single Hausdorff point,
if $\Lambda^{1,1}(p)$ is trivial, or if $\Lambda^{1,1}(p)$ 
is of rank $1$, generated by a class $\alpha$ satisfying $(\alpha,\alpha)\geq 0$.
\end{enumerate}
\end{thm}
}

Let $\C_\Lambda:=\{x\in \Lambda\otimes_\Integers\RealNumbers \ : \ (x,x)>0\}$ be the positive cone.
Then $H^2(\C_\Lambda,\Integers)$ is isomorphic to $\Integers$ and is thus a character of the isometry group
$O(\Lambda)$, called the {\em spinor norm}, or {\em orientation} of the positive cone \cite[Lemma 4.1]{markman-survey}.
Denote by $O^+(\Lambda)$ the subgroup of isometries preserving the orientation of $\C_\Lambda$.
Let $W$ be the subgroup of $O^+(\Lambda)$ acting on
$\Lambda^*/\Lambda$ by multiplication by $\pm 1$. 
Let $\Sigma$ be the subset of $\Lambda$ consisting of primitive classes $\delta$ satisfying 
$(\delta,\delta)=2-2n$, and such that $(\delta,\lambda)$  is divisible by $2n-2$, for all $\lambda\in \Lambda$. 
A class $\delta$ of $\Lambda$ belongs to $\Sigma$, if and only if $\delta^\perp$
is isometric to $\Lambda_{K3}$, where 
$\delta^\perp$ is the sublattice 
orthogonal to $\delta$ in $\Lambda$ \cite[Lemma 7.1]{markman-prime-exceptional}.

Given a $W$-orbit $\Sigma'$ in $\Sigma$, set 
\[
\Omega_{\Sigma'} \ \ := \ \ 
\{p\in \Omega \ : \ \Lambda^{1,1}(p)\cap \Sigma' \neq \emptyset
\}.
\]

We prove in this section the following statement.

\begin{lem}
\label{lemma-Omega-orbit-in-Sigma-is-dense}
The subset $\Omega_{\Sigma'}$ is dense in $\Omega$, for every $W$-orbit $\Sigma'$.
\end{lem}

The proof will require two Lemmas. 
Set $\widetilde{\Lambda}  :=  \Lambda_{K3} \oplus H.$
Choose a primitive embedding
\begin{equation}
\label{eq-iota}
\iota \ : \ \Lambda \ \ \ \hookrightarrow \ \ \ \widetilde{\Lambda}.
\end{equation}
Given a class $\delta$ in $\Sigma$ set $H_{\iota,\delta}:=[\iota(\delta^\perp)]^\perp$.
Then $H_{\iota,\delta}$ is isometric to $H$.
Let $v$ be a class generating the sub-lattice in $\widetilde{\Lambda}$ orthogonal to 
$\iota(\Lambda)$. Then $H_{\iota,\delta}$ is the saturation in $\widetilde{\Lambda}$ 
of the sub-lattice spanned by $\iota(\delta)$ and $v$.

Let $I_n(H)$ be the subset of $H$ consisting of primitive classes $\delta$, such that $(\delta,\delta)=2-2n$.
The isometry group $O(H)$ is isomorphic to $\Integers/{2\Integers}\times \Integers/{2\Integers}$.
Let $I_n(H)/O(H)$ be the orbit space. This orbit space is in natural bijection with isometry
classes $[(H',\delta')]$ of pairs $(H',\delta')$, where $H'$ is a lattice isometric to $H$ and
$\delta'$ is a primitive class in $H'$ satisfying $(\delta',\delta')=2-2n$. 
Indeed, given such a pair $(H',\delta')$, choose an isometry $g:H'\rightarrow H$.
We get the orbit $O(H)g(\delta')$, which is independent of the isometry $g$ chosen.
Define
\[
f \ : \ \Sigma \ \ \ \rightarrow \ \ \ I_n(H)/O(H),
\]
by $f(\delta):=[\left(H_{\iota,\delta},\iota(\delta)\right)]$.

\begin{lem} 
\label{lemma-monodromy-invariants}
\cite[Lemma 6.4]{markman-prime-exceptional}
Two classes $\delta_1$, $\delta_2$ in $\Sigma$ belong to the same $W$-orbit in $\Sigma$,
if and only if $f(\delta_1)=f(\delta_2)$. The map $f$ is surjective.
\end{lem}

\begin{rem}
The cardinality of $I_n(H)/O(H)$ is
$2^{\rho-1}$, where $\rho$ is the number of distinct positive primes in the
prime decomposition $n-1=p_1^{e_1} \cdots p_\rho^{e_\rho}$
\cite[Lemmas 4.2 and 4.3]{markman-constraints}.
\end{rem}

\hide{
\begin{lem} 
\cite[Lemmas 4.2 and 4.3]{markman-constraints}
\begin{enumerate}
\item
The group $O^+(\Lambda)$ acts transitively on $\Sigma$ and $\Sigma$ contains 
$2^{\rho-1}$  $W$-orbits, where $\rho$ is the number of distinct positive primes in the
prime decomposition $n-1=p_1^{e_1} \cdots p_\rho^{e_\rho}$.
\item
Let $\iota:\Lambda\rightarrow \widetilde{\Lambda}$ be a primitive embedding
and $h$ an element of $O^+(\Lambda)$. There exists an isometry $g$ of $O(\widetilde{\Lambda})$,
satisfying $\iota\circ h=g\circ \iota$, if and only if $h$ belongs to $W$.
\end{enumerate}
\end{lem} 
}

Let $q\in \Omega$ be a period, such that $\Lambda^{1,1}(q)$ has rank $21$ and signature $(1,20)$.
Denote by
$T_q$ the rank $2$ sub-lattice of $\Lambda$ orthogonal to $\Lambda^{1,1}(q)$. 
Note that $T_q$ is positive definite. 
Let $v$ be a primitive class in $\widetilde{\Lambda}$ orthogonal to $\iota(\Lambda)$. Then $(v,v)=2n-2$. 
Denote by $L_{q,\iota}$ the saturation in $\widetilde{\Lambda}$ of the lattice spanned by 
$\iota(T_q)$ and $v$. 
Let $A$ be the set of periods $q\in \Omega$, 
such that $\Lambda^{1,1}(q)$ has rank $21$ and signature $(1,20)$, and $\iota(T_q)+\Integers v$ is saturated
in $\widetilde{\Lambda}$ and so equal to $L_{q,\iota}$.

\begin{lem}
\label{lemma-A-is-dense}
The set $A$ is dense in $\Omega$.
\end{lem}

\begin{proof}
The period domain $\Omega$ may be identified with
the Grassmannian of positive oriented 2-planes in $\Lambda\otimes_\ZZ \RR$
\cite[VIII.8]{BHPV}. Thus given
$q \in \Omega$ with $\Lambda^{1,1}(q)$ of rank $21$, we want to produce
a positive definite rank $2$ sublattice $\Pi \subset \iota(\Lambda)$ 
arbitrarily close to $\iota(T_q)$, such that the span of $\Pi$ and 
$v$ is saturated. The idea is to do so by perturbing 
$\iota(T_q)$, keeping $v$ fixed.

By \cite[Theorem I.2.9, (ii)]{BHPV}, the embedding of 
$L_{\iota,q}$ into $\widetilde{\Lambda}$
is unique up to the action of $O(\widetilde{\Lambda})$,
while by part (i) of the same result,  $L_{\iota,q}$ embeds into $H^{\oplus 3}$.
Thus, $L_{\iota,q}^{\perp}\subset\widetilde{\Lambda}$ contains $H$ as a direct
summand, $H=\langle e_1,f_1 \rangle$, with $(e_1,e_1)=(f_1,f_1)=0$ and $(e_1,f_1)=1$.

Suppose $\iota(T_q)$ is spanned by two elements
$u_1,u_2\in \iota(\Lambda)$. 
For $k\in \NN$, define $T_1$ to be the rank 2 sublattice of 
$\widetilde{\Lambda}$ spanned by $u_1':=ku_1 + e_1$ and $u_2$. Then, 
$T_1$ is positive definite because $e_1\perp T_q$ and $T_q$ is positive
definite. Also $(u_1',v)=0$, whence
$T_1\subset \iota(\Lambda)$. Choosing $k$ large enough,
we may arrange $u_1'$ to be as close to $u_1$ in 
$\PP (\widetilde{\Lambda} \otimes \RR)$ as desired. 
Also note that $(f_1,v)=(f_1,u_2)=0$, and $(f_1,u_1')=1$.  

The reasoning of the previous paragraph, now applied to the 
sublattice $T_1$ and its element $u_2\in T_1$, 
produces an element $u_2'$ arbitrarily close to $u_2$, such 
that $\Pi:=\langle u_1', u_2'\rangle$ is positive definite, and
$\Pi\subset v^\perp=\iota(\Lambda)$.
Moreover, as above, there exist an element $f_2$ such that $(f_2,v)=(f_2,u_1')=0$,
and $(f_2,u_2')=1$. Thus we have a pair of elements
$f_1,f_2$ in $\widetilde{\Lambda}$ such that $(f_i,v)=0$, and the matrix $(f_i,u_j')$ is upper-triangular with
1's on the diagonal. 
Let $L$ be the saturation of $\Pi\oplus\Integers v$ in $\widetilde{\Lambda}$.
Then $(f_1,\bullet)$ and $(f_2,\bullet)$ restrict to elements of $\ker[L^*\rightarrow (\Integers v)^*]$
and their images in $\Pi^*$ span the latter.
This shows that $L=\Pi\oplus \Integers v$.
\end{proof}

\begin{proof} (Of Lemma \ref{lemma-Omega-orbit-in-Sigma-is-dense}).
It suffices to show  that $\Omega_{\Sigma'}$ contains the set $A$,  
by Lemma \ref{lemma-A-is-dense}.
Let $q$ be a period in $A$. 
Let $\widetilde{L}_q$ be the rank $4$ lattice, which is the orthogonal direct sum of $T_q$ and $H$.
There exists a primitive embedding 
$\tilde{\iota}_1: \widetilde{L}_q\rightarrow \widetilde{\Lambda}$, by \cite[Theorem I.2.9]{BHPV}.
Choose a primitive class $v_1$ in $\tilde{\iota}_1(H)$ satisfying $(v_1,v_1)=2n-2$. 
Let $\tilde{\delta}_1$ be a primitive class in $\tilde{\iota}_1(H)$ orthogonal to $v_1$.
Then $(\tilde{\delta}_1,\tilde{\delta}_1)=2-2n$, and $\tilde{\delta}_1$ is orthogonal to $\tilde{\iota}_1(T_q)$ and to $v_1$.
Furthermore, $(\tilde{\delta}_1,\lambda)$ is divisible by $2n-2$, for every class $\lambda$ of $\widetilde{\Lambda}$
orthogonal to $v_1$. Let $\iota_1:L_{q,\iota}\rightarrow \widetilde{\Lambda}$ be the restriction of 
$\tilde{\iota}_1$ to $T_q \oplus \Integers v_1$ composed with the 
inverse of the isometry 
$T_q\oplus \Integers v_1\rightarrow L_{q,\iota}=\iota(T_q)\oplus \Integers v$, restricting to $\iota$ on $T_q$ and 
sending $v_1$ to $v$.

Let $\iota_2:L_{q,\iota}\rightarrow \widetilde{\Lambda}$ be the given inclusion.
There exists an isometry $g\in O(\widetilde{\Lambda})$ satisfying 
\[
g\circ \iota_1 = \iota_2,
\] 
by \cite[Theorem I.2.9]{BHPV}. 
Set $\tilde{\iota}_2:=g\circ \tilde{\iota}_1$ and $\tilde{\delta}_2:=g(\tilde{\delta}_1)$.
Note that $\tilde{\delta}_2$ is orthogonal to $L_{q,\iota}$, since $L_{q,\iota}=\iota_2(L_{q,\iota})$, and $\tilde{\delta}_1$ is orthogonal to 
$\iota_1(L_{q,\iota})$.
Set  $\delta:=\iota^{-1}(\tilde{\delta}_2)$, where the embedding $\iota$ is given in (\ref{eq-iota}).
Then $\delta$ belongs to $\Lambda^{1,1}(q)$ and to $\Sigma$. Consequently,
$q$ belongs to $\delta^\perp\cap \Omega$ and so $q$ belongs to $\Omega_{\Sigma'}$,
for some $W$-orbit $\Sigma'$ in $\Sigma$.

\hide{
Let $\rho$ be the largest positive integer, such that $(\tilde{\delta}_1+v_1)/\rho$ is an integral class
of $\widetilde{\Lambda}$. Let $\sigma$ be the largest positive integer, such that 
$(\tilde{\delta}_1-v_1)/\sigma$ is an integral class of $\widetilde{\Lambda}$. 
Notice that $\rho$ is also the the largest positive integer, such that $(\tilde{\delta}_2+v)/\rho$ is an integral class
of $\widetilde{\Lambda}$. Similarly, $\sigma$ is the largest positive integer, such that 
$(\tilde{\delta}_2-v)/\sigma$ is an integral class of $\widetilde{\Lambda}$. 
The $W$-orbit $\Sigma'$ of $\delta$ is uniquely determined by the unordered pair 
$\{\rho,\sigma\}$, by \cite[Prop. 9.16]{markman-survey}. 
}

As we vary the choice of $v_1$ in $\tilde{\iota}_1(H)$, in the above construction, 
all possible values of the invariant $f:\Sigma\rightarrow I_n(H)/O(H)$ are obtained.
Hence, the above construction produces a class $\delta$
in every $W$-orbit $\Sigma'$ in $\Sigma$, by Lemma \ref{lemma-monodromy-invariants}.
\end{proof}

%
\section{Density in moduli}
\label{sec-density-in-moduli}
We prove Theorem \ref{thm-density-in-moduli} in this section. 
\begin{defi}
Let $X_1$ and $X_2$ be irreducible holomorphic symplectic manifolds. 
An isomorphism $g:H^2(X_1,\Integers)\rightarrow H^2(X_2,\Integers)$ 
is called a {\em parallel-transport operator}, if there exists a family 
$\pi:\X\rightarrow B$ (which may depend on $g$) of irreducible holomorphic symplectic manifolds
over an analytic base $B$, points $b_1$ and $b_2$ in $B$, 
isomorphisms $\psi_i:\X_{b_i}\rightarrow X_i$, $i=1,2$, and a continuous path
$\gamma$ from $b_1$ to $b_2$, such that parallel-transport along $\gamma$
in the local system $R^2\pi_*\Integers$ induces the isomorphism
$\psi_{2}^{*}\circ g\circ \psi_{1_*}$.
\end{defi}

A parallel transport operator is necessarily an isometry with respect to the
Beauville-Bogomolov forms, due to the topological nature of these forms
\cite{beauville}. 
The positive cone $\widetilde{\C}_{X_i}:=\{x\in H^2(X_i,\RealNumbers) \ : \ (x,x)>0\}$ is homotopic to the $2$-sphere,
and $H^2(\widetilde{\C}_{X_i},\Integers)$ comes with a natural generator, called the
{\em orientation class}
\cite[Sec. 4]{markman-survey}. An isometry
$g:H^2(X_1,\Integers)\rightarrow H^2(X_2,\Integers)$ is {\em orientation preserving},
if the induced map $\bar{g}: \widetilde{\C}_{X_1}\rightarrow \widetilde{\C}_{X_2}$ is.

The group $Mon^2(X)$ 
of parallel-transport operators from $H^2(X,\Integers)$ to itself is called the
{\em monodromy group}. Let $S$ be a $K3$ surface and $S^{[n]}$ its Hilbert scheme. Denote by 
$W(S^{[n]})$ the subgroup of the orientation-preserving isometry group of 
$H^2(S^{[n]},\Integers)$ consisting of elements acting by $\pm 1$ on
$H^2(S^{[n]},\Integers)^*/H^2(S^{[n]},\Integers)$. 
Given a marking $\eta:H^2(S^{[n]},\Integers)\rightarrow \Lambda$, 
we get the equality $W(S^{[n]})=\eta^{-1}W\eta$.

\begin{thm}\cite[Theorem 1.2 and Lemma 4.2]{markman-constraints}
\label{thm-monodromy}
The monodromy group $Mon^2(S^{[n]})$  is equal to $W(S^{[n]})$.
\end{thm}

Let $S$ be a $K3$ surface, $Z\subset S\times S^{[n]}$ the universal subscheme, and
$I_Z$ its ideal sheaf. The {\em Mukai lattice} $\widetilde{H}(S,\Integers)$ of $S$ 
is the integral cohomology group $H^*(X,\Integers)$ endowed with the Mukai pairing, which we now
recall. A class $\lambda$ in $\widetilde{H}(S,\Integers)$
decomposes as $\lambda=(\lambda_0,\lambda_1,\lambda_2)$, with $\lambda_i\in H^{2i}(S,\Integers)$.
The {\em Mukai pairing} is given by
\[
(\lambda,\lambda') \ \ := \ \ \int_S\left(-\lambda_0\lambda'_2+\lambda_1\lambda'_1-\lambda_2\lambda_0'\right).
\]
The Mukai lattice is isometric to $\widetilde{\Lambda}$ \cite{mukai-hodge}. 
The {\em Mukai vector} of a coherent sheaf $F$ on $S$ is $v(F):=ch(F)\sqrt{td_S}$. 
Identify $H^0(S,\Integers)$ and $H^4(S,\Integers)$ with $\Integers$ using the classes
Poincare dual to $S$ and to the class of a point. 
Then $v(F)=(r,c_1(F),s)$, where $r=\rank(F)$, and $s=\chi(F)-r$, by the Riemann-Roch Theorem.

Let $v:=(1,0,1-n)$ be the Mukai vector in $\widetilde{H}(S,\Integers)$ of the ideal sheaf
of a length $n$ subscheme. Denote by $v^\perp$ the sub-lattice of $\widetilde{H}(S,\Integers)$
orthogonal to $v$. Let $\pi_i$ be the projection from $S\times S^{[n]}$ onto the $i$-th factor, $i=1,2$. 
We get Mukai's isometry
\[
\theta \ : \ v^\perp \ \ \rightarrow \ \ H^2(S^{[n]},\Integers),
\]
sending a Mukai vector $\lambda$ to the degree $2$ summand of
$\pi_{2_*}[\pi_1^*(\lambda^\vee\sqrt{td_S}) ch(I_Z)]$, 
where $\lambda^\vee:=(\lambda_0,-\lambda_1\lambda_2)$ (see \cite{ogrady-weight-two}).
Note that  $v^\perp$ is the orthogonal direct sum of the sub-lattice
$H^2(S,\Integers)$ of $\widetilde{H}(S,\Integers)$ and $\Integers\tilde{\delta}$, where 
$\theta(\tilde{\delta})$ is half the class of the diagonal divisor in $S^{[n]}$.

\begin{lem}
\label{lemma-signed-isometry-mapping-delta-to-delta-is-a-parallel-transport}
Let $S_1$ and $S_2$ be $K3$ surfaces.
Denote by $\delta_i$ the class in $H^2(S_i^{[n]},\Integers)$, which is half the class of the
diagonal divisor in $S_i^{[n]}$.
Let $g:H^2(S_1^{[n]},\Integers)\rightarrow H^2(S_2^{[n]},\Integers)$ be an orientation preserving isometry, 
satisfying $g(\delta_1)=\delta_2$. Then $g$ is a parallel-transport operator.
\end{lem}

\begin{proof}
The positive cone in $\delta_i^\perp\cap H^2(S^{[n]}_i,\RealNumbers)$ is again homotopic to the $2$-sphere
and the orientation class in $H^2(\C_{S^{[n]}_i},\Integers)$ restricts to an orientation class on the former cone.
The isometry $g$ restricts to an orientation preserving isometry 
$h:\delta_1^\perp\rightarrow \delta_2^\perp.$ 
Let $v_i\in \widetilde{H}(S_i,\Integers)$ be the Mukai vector of the ideal sheaf of a length $n$ subscheme.
We have the equality 
$\delta_i^\perp=\theta_i[H^2(S_i,\Integers)]$, 
where $\theta_i:v_i^\perp\rightarrow H^2(S_i^{[n]},\Integers)$ is the Mukai isometry. 
Let $\tilde{\delta}_i$ be the class in $v_i^\perp$ satisfying $\theta_i(\tilde{\delta}_i)=\delta_i$.
We see that $\theta_2^{-1}g\theta_1:v_1^\perp\rightarrow v_2^\perp$ is an isometry
$[H^2(S_1,\Integers)\oplus \Integers\tilde{\delta}_1]\rightarrow [H^2(S_2,\Integers)\oplus \Integers\tilde{\delta}_2]$ 
mapping $\tilde{\delta}_1$ to $\tilde{\delta}_2$ and mapping 
$H^2(S_1,\Integers)$ to $H^2(S_2,\Integers)$ via an orientation preserving isometry $h'$ conjugate to $h$.
Every orientation preserving isometry $h':H^2(S_1,\Integers)\rightarrow H^2(S_2,\Integers)$ is a 
parallel-transport operator, by
\cite{borcea}. 
Hence, so is $g$, since the construction of Hilbert schemes (and Douady spaces)
works in families and lifts parallel-transport operators between $K3$-surfaces
to parallel-transport operators between their Hilbert schemes.
\end{proof}

Let $S_0$ be a $K3$ surface and $(S_0^{[n]},\eta_0)$ a marked pair in $\fM^0_\Lambda$.
Such a marked pair exists, by our choice of the component $\fM^0_\Lambda$.
Let $\delta_0\in H^2(S_0^{[n]},\Integers)$ be half the class of the diagonal divisor. 
Let $\Sigma'\subset \Sigma$ be the $W$-orbit of $\eta_0(\delta_0)$. 

\begin{lem}
\label{lemma-X-is-bimeromorphic-to-a-Hilbert-scheme}
Let $(X,\eta)$ be a marked pair in $\fM^0_\Lambda$, such that its period belongs to $\Omega_{\Sigma'}$.
Then $X$ is bimeromorphic to the Hilbert scheme $S^{[n]}$ of some K\"{a}hler $K3$ surface $S$.
\end{lem}

\begin{proof}
There exists a class $\delta'\in \Sigma'$, such that $P(X,\eta)$ is orthogonal to $\delta'$,
by our assumption on $(X,\eta)$. Set $\delta:=\eta^{-1}(\delta')$. Then $\delta$ is of Hodge type
$(1,1)$ and the sub-lattice $\delta^\perp\subset H^2(X,\Integers)$ is isometric to the $K3$ lattice
$\Lambda_{K3}$. Hence, there exists a $K3$ surface $S$ and a Hodge isometry
$\gamma:H^2(S,\Integers)\rightarrow \delta^\perp$, by the surjectivity of the period map for K\"{a}hler $K3$ surfaces
\cite{siu,todorov}. 
Let $\delta_1\in H^2(S^{[n]},\Integers)$ be half the class of the diagonal divisor.
Extend $\gamma$ to a Hodge isometry 
\begin{equation}
\label{eq-gamma-1}
\gamma_1 \ : \ H^2(S^{[n]},\Integers) \ \ \ \rightarrow \ \ \  H^2(X,\Integers),
\end{equation}
by sending $\delta_1$ to $\delta$. 
We may assume that $\gamma_1$ is orientation preserving, possibly after replacing $\gamma$ with $-\gamma$ 
in the above construction.

We prove next that $\gamma_1$ is a parallel-transport operator. 
Set 
\[
\psi:=\eta_0^{-1}\eta\circ \gamma_1:H^2(S^{[n]},\Integers)\rightarrow H^2(S_0^{[n]},\Integers).
\]
Now $\psi(\delta_1)=\eta_0^{-1}\eta(\delta)$
belongs to the $W(S^{[n]}_0)$-orbit of $\delta_0$, since $\eta_0(\delta_0)$ and $\eta(\delta)$ both belong to the 
$W$-orbit $\Sigma'$.
Let $w$ be an element of $W(S_0^{[n]})$ satisfying 
$w\psi(\delta_1)=\delta_0$. Then $w\psi$ is a parallel-transport operator, by
Lemma \ref{lemma-signed-isometry-mapping-delta-to-delta-is-a-parallel-transport}. 
Now $w$ is a parallel-transport operator, by Theorem \ref{thm-monodromy}.
Hence, $\psi$ is a parallel-transport operator as well. We know that $\eta_0^{-1}\eta$
is a parallel-transport operator, since the marked pairs $(S_0^{[n]},\eta_0)$ and
$(X,\eta)$ belong to the same connected component $\fM^0_\Lambda$. Hence, 
$\gamma_1$ is a parallel-transport operator.

Verbitsky's Hodge theoretic Torelli Theorem states that two irreducible holomorphic symplectic 
manifolds $X$ and $Y$ are bimeromorphic, if and only if there exists 
a parallel-transport operator from $H^2(X,\Integers)$ to $H^2(Y,\Integers)$,
which is an isomorphism of Hodge structures (see \cite{verbitsky,huybrechts-bourbaki} and  
\cite[Theorem 1.3]{markman-survey} for this specific statement). Hence, $X$ and $S^{[n]}$ 
are bimeromorphic. 
\end{proof}

\begin{proof} (Of Theorem \ref{thm-density-in-moduli}).
Let $B$ be the subset of $\Omega_{\Sigma'}$ consisting of periods $p$,
such that $\Lambda^{1,1}(p)$ has rank $1$. Then $B$ is dense in $\Omega_{\Sigma'}$
and hence also in the period domain $\Omega$, by Lemma \ref{lemma-Omega-orbit-in-Sigma-is-dense}. 
The Local  Torelli Theorem implies that $P^{-1}(B)$ is a dense  subset of
$\fM_\Lambda$. 

Let $p$ be a point of $B$. 
Then $\Lambda^{1,1}(p)$ is spanned by a class $\delta'\in\Sigma'$. 
Hence, every marked pair $(X,\eta)$ in the
fiber of the period map $P:\fM^0_\Lambda\rightarrow \Omega_\Lambda$
over $p$ satisfies $H^{1,1}(X,\Integers)=\Integers \eta^{-1}(\delta')$.
Set $\delta:=\eta^{-1}(\delta')$. 
Precisely one of the classes $2\delta$ or $-2\delta$ is effective, since $X$ is bimeromorphic to $S^{[n]}$,
for some $K3$ surface $S$, by Lemma \ref{lemma-X-is-bimeromorphic-to-a-Hilbert-scheme}. 
The classes $\delta'$ and $-\delta'$ belong to the same $W$-orbit, since the reflection
$R_{\delta'}$, given by $R_{\delta'}(\lambda)=\lambda-2\frac{(\lambda,\delta')}{(\delta',\delta')}\delta'$,
belongs to $W$ and $R_{\delta'}(\delta')=-\delta'$.
We may assume that $2\delta$ is effective, possibly after replacing $\delta'$ by $-\delta'$ 
above. 
A class $\alpha$ in the positive cone of $X$ is  a K\"{a}hler class, 
if and only if $(\alpha,\delta)>0$,  since $H^{1,1}(X,\RationalNumbers)$
has rank $1$, and it is spanned by the class of an effective divisor
\cite[Theorem 4.3]{huybrechts-kahler-cone,boucksom-kahler-cone}. 

Let $\kappa$ be a K\"{a}hler class on $S^{[n]}$ and 
$\gamma_1:H^2(S^{[n]},\Integers) \rightarrow H^2(X,\Integers)$ the parallel-transport Hodge isomorphism 
constructed in equation (\ref{eq-gamma-1}). Let $R_{\delta_1}:H^2(S^{[n]},\Integers)\rightarrow H^2(S^{[n]},\Integers)$
be the reflection with respect to the class $\delta_1$ (half the class of the diagonal divisor). 
Let $R_\delta:H^2(X,\Integers)\rightarrow H^2(X,\Integers)$ be the reflection by $\delta$.
Then $R_\delta(\delta)=-\delta$ and $R_\delta\gamma_1=\gamma_1R_{\delta_1}$.
We see that 
one of $\gamma_1(\kappa)$ or $\gamma_1R_{\delta_1}(\kappa)$ is a K\"{a}hler class on $X$.
Now $R_{\delta_1}$ belongs to $W(S^{[n]})$ and is thus a parallel-transport operator,
by Theorem \ref{thm-monodromy}. Hence, 
there exists a parallel-transport operator from $H^2(S^{[n]},\Integers)$ to $H^2(X,\Integers)$,
which is an isomorphism of Hodge structures, and which
maps a K\"{a}hler class on $S^{[n]}$ to a K\"{a}hler class on $X$. Thus, $X$ is isomorphic to $S^{[n]}$,
by the Hodge-theoretic Torelli Theorem \cite[Theorem 1.3]{markman-survey}.

Let $Def(S)$ be the Kuranishi deformation space of $S$ and ${\mathfrak S}\rightarrow Def(S)$ a universal family.
We assume that $Def(S)$ is simply connected, possibly after replacing it by an open neighborhood
of the point $0\in Def(S)$ parametrizing $S$. The isomorphism $X\cong S^{[n]}$ and the 
relative Hilbert scheme ${\mathfrak S}^{[n]}\rightarrow Def(S)$
induce a natural open morphism $Def(S)\rightarrow \left[\fM^0_\Lambda\cap P^{-1}((\delta')^\perp)\right]$,
mapping $0$ to $(X,\eta)$. The locus of projective $K3$ surfaces is dense in $Def(S)$.
Hence,  any open neighborhood of $(X,\eta)$ in 
$\fM^0_\Lambda$ contains a marked projective Hilbert scheme ${\mathfrak S}^{[n]}_t$, 
for some $t\in Def(S)$.
\end{proof}

%
\section{Density of generalized Kummer varieties}
\label{sec-density-of-kummers}

Let $S$ be a two dimensional compact complex torus, $s_0\in S$ its origin, $n\geq 2$ an integer,
$S^{(n+1)}$ the $n+1$ symmetric product of $S$, and $S^{(n+1)}\rightarrow S$ the summation morphism.
Consider the composition $\pi:S^{[n+1]}\rightarrow S$ of the Hilbert-Chow 
morphism $S^{[n+1]}\rightarrow S^{(n+1)}$ with the summation morphism.
The fiber $K^{[n]}(S)$ of $\pi$ over $s_0$ is a
$2n$-dimensional irreducible holomorphic symplectic manifold, called
a {\em generalized Kummer variety} \cite{beauville}.

Let $\Lambda$ be the lattice $H\oplus H \oplus H \oplus \langle -2-2n\rangle$.
Let $X$ be an irreducible holomorphic symplectic manifold deformation equivalent to 
$K^{[n]}(S)$. Then $H^2(X,\Integers)$, endowed with its Beauville-Bogomolov form, 
is isometric to $\Lambda$, by \cite{beauville}. 
Let $\fM^0_\Lambda$ be a connected component of the moduli space $\fM_\Lambda$, 
containing a marked pair $(K^{[n]}(S_0),\eta_0)$.

\begin{thm}
\label{thm-density-of-kummers}
The locus in $\fM^0_\Lambda$, consisting of marked pairs $(X,\eta)$, where $X$ is isomorphic
to the generalized Kummer variety
$K^{[n]}(S)$, for some abelian surface $S$, is dense in $\fM^0_\Lambda$.
\end{thm}

Theorem \ref{thm-density-of-kummers} is proven at the end of this section.
The proof will depend on 
the following statement.  
Let $\W$ be the subgroup of $O^+(\Lambda)$ acting by $\pm 1$ on
$\Lambda^*/\Lambda$. Denote by $\chi:\W\rightarrow \{\pm 1\}$
the character corresponding to the action on $\Lambda^*/\Lambda$
and let $\det:\W\rightarrow \{\pm 1\}$ be the determinant character.
Let $\N$ be the kernel of the  product character $(\det\cdot \chi):\W\rightarrow \{\pm 1\}$.
Note that $\N$ is a normal subgroup of $O(\Lambda)$. 
Hence, $\N$ determines a well defined subgroup of any lattice isometric to $\Lambda$.
Denote by 
\[
\N(X) 
\]
the corresponding normal subgroup of the isometry group of $H^2(X,\Integers)$.

\begin{thm}
\label{thm-monodromy-of-kummers}
The subgroup $\N(X)$ is contained in the monodromy group $Mon^2(X)$.
\end{thm}

Theorem \ref{thm-monodromy-of-kummers} was proven by the first author. 
Its proof   is similar to that 
of \cite[Corollary 1.8]{markman-monodromy-I} and will appear elsewhere.
The proof relies on examples of Yoshioka of 
stability preserving Fourier-Mukai transformations
between abelian surfaces \cite{yoshioka-abelian-surface}.

Let $\Sigma\subset \Lambda$ be the subset 
consisting of primitive classes $\delta$ satisfying $(\delta,\delta)=-2-2n$
and such that $(\delta,\lambda)$ is divisible by $2n+2$, for every class $\lambda$ in $\Lambda$.
Let $\Lambda_T$ be the orthogonal direct sum  of three copies of $H$.
The class $\delta$ belongs to $\Sigma$, if and only if $\delta^\perp$ is isomorphic to $\Lambda_T$
\cite[Lemma 7.1]{markman-prime-exceptional}. Hence, the reflection
$R_\delta\in O(\Lambda)$ acts by $-1$ on $\Lambda^*/\Lambda$, and so $R_\delta$
belongs to $\N$.

Let $\widetilde{\Lambda}$ be the orthogonal direct sum of four copies of $H$.
Choose a primitive embedding $\iota:\Lambda\hookrightarrow \widetilde{\Lambda}$.
Given a class $\delta$ in $\Sigma$ denote again by $H_{\iota,\delta}$ the saturation
in $\widetilde{\Lambda}$ of the sublattice generated by $\iota(\delta)$ and $\iota(\Lambda)^\perp$.
Again $H_{\iota,\delta}$ is isometric to $H$.
Let $J_n$ be the subset of $H$ consisting of primitive classes $\delta$, such that $(\delta,\delta)=-2-2n$.
Let $J_n/O(H)$ be the orbit set. Define
\[
f \ : \ \Sigma \ \ \ \rightarrow \ \ \ J_n/O(H)
\]
by sending $\delta\in\Sigma$ to the isometry class of $(H_{\iota,\delta},\iota(\delta))$.

\begin{lem}
\label{lemma-monodromy-invariants-for-kummers}
Two classes $\delta_1$ and $\delta_2$ in $\Sigma$ belong to the same $\N$-orbit, 
if and only if $f(\delta_1)=f(\delta_2)$. The map $f$ is surjective.
\end{lem}

\begin{proof}
$\N$ is an index $2$ subgroup of $\W$.
Each $\W$-orbit in $\Sigma$ is also an $\N$-orbit. Indeed, given $\delta\in \Sigma$
and $w\in \W\setminus\N$,
choose  a class $e\in \delta^\perp$ satisfying $(e,e)=-2$. Then $w\circ R_e$ belongs to $\N$
and $w(R_e(\delta))=w(\delta)$. The proof now is identical to that of 
Lemma \ref{lemma-monodromy-invariants}.
\end{proof}

%

Let $\Sigma'$ be an $\N$-orbit in $\Sigma$.
Define  the set $\Omega_{\Sigma'}$ exactly the same way the set $\Omega_{\Sigma'}$ 
was defined in section \ref{sec-density-of-periods}.
Lemma \ref{lemma-Omega-orbit-in-Sigma-is-dense},
stating that $\Omega_{\Sigma'}$ is dense in $\Omega$, now holds for every $\N$-orbit $\Sigma'$ instead of every
$W$-orbit.
The proof is identical, replacing 
the use of Lemma \ref{lemma-monodromy-invariants} by that of 
Lemma \ref{lemma-monodromy-invariants-for-kummers}.

Let $V$ be a free abelian group of rank $4$, $V^*$ the dual group, $\omega\in \ \Wedge{4}V$ a generator
and $\omega^*\in \ \Wedge{4}V^*$ the dual generator. 
Endow $\Wedge{2}V$ with the pairing $(\alpha,\beta)=\omega^*(\alpha\wedge\beta)$.
Then $\Wedge{2}V$ is isometric to the orthogonal direct sum of three copies of $H$. Let 
\begin{equation}
\label{eq-phi}
\phi \ : \ \ \Wedge{2}V \ \ \rightarrow  \ \ \Wedge{2}V^*
\end{equation}
be the isomorphism given by $\phi(\alpha)(\bullet)=\omega^*(\alpha\wedge\bullet)$.

\begin{defi}
Given two free abelian group $V_1$, $V_2$ of rank $4$, with bases $\omega_i\in \ \Wedge{4}V_i$, 
and an isometry $g:\ \Wedge{2}V_1\rightarrow \Wedge{2}V_2$, we say that $g$ 
{\em preserves the orientation of the positive (resp. negative) cones}, if there exists an isomorphism
$h:V_2\rightarrow V_1$, satisfying $(\Wedge{4}h)(\omega_2)=\omega_1$, such that 
$(\Wedge{2}h)g$ preserves the orientation of the positive (resp. negative) cone in $V_1$.
\end{defi}

The terms above are well defined, since the image of the homomorphism
\[
\Wedge{2} \ : \  SL(V_1) \ \ \rightarrow \ \ \ O\left[\Wedge{2}V_1\right]
\]
is the subgroup $SO^+(\Wedge{2}V_1)$ of isometries preserving the orientation of both cones.

\begin{lem}
\label{lemma-phi-does-not-lift}
The homomorphism $\phi$, given in (\ref{eq-phi}), 
is an isometry which preserves the orientation of the negative cones, but reverses
the orientations of the positive cones. In particular, 
there does not exist an isomorphism $h:V\rightarrow V^*,$
such that $\Wedge{2}h=\phi$.
\end{lem}

\begin{proof}
Let us first sketch the conceptual explanation.
Let $\varphi:SL(V)\rightarrow SL(V^*)$ be the natural isomorphism.
We have the equality  $\Wedge{2}\varphi=Ad_\phi:SO^+(\Wedge{2}V)\rightarrow SO^+(\Wedge{2}V^*)$.  
There does not exist an isomorphism $h:V\rightarrow V^*$, such that $Ad_h=\varphi$, since
$V$ and $V^*$ are two distinct representations of $SL(V)$. 
Hence, there does not exist such an $h$, satisfying
$Ad_{\Wedge{2}h}=Ad_\phi$, since $V$ and $V^*$ are distinct half spin representations of the spin group
$SL(V)$ of $SO^+(V)$.

We provide next an explicit proof, which will determine the affect of $\phi$ on 
the orientation of each cone.
Let $\{e_1,e_2,e_3,e_4\}$ be a basis of $V$, such that $\omega=e_1\wedge e_2\wedge e_3\wedge e_4$.
Denote by $\{e_1^*,e_2^*,e_3^*,e_4^*\}$ the dual basis of $V^*$. 
Let $f:V^*\rightarrow V$ be the isomorphism sending $e_i^*$ to $e_i$.
For every oriented isomorphism $h:V\rightarrow V^*$, as in the Lemma, 
$\Wedge{2}(fh)$ belongs to the image $SO^+(\Wedge{2}V)$ of $SL(V)$.
Set $\psi:=(\Wedge{2}f)\phi$. 
We determine the affect of $\psi$ on the orientations of the positive and negative cones.

The values of $\psi$ are calculated in the following table.
\[
\begin{array}{cccccccc}
\alpha & : &e_1\wedge e_2 & e_3\wedge e_4 & e_1\wedge e_3 & e_2\wedge e_4 & e_1\wedge e_4 & e_2\wedge e_3
\\
\psi(\alpha) &:& e_3\wedge e_4 & e_1\wedge e_2 & -e_2\wedge e_4 & -e_1\wedge e_3 & e_2\wedge e_3 & e_1\wedge e_4.
\end{array}
\]
Set 
$\alpha=e_1\wedge e_2 - e_3\wedge e_4$,
$\beta=e_1\wedge e_3 + e_2\wedge e_4$,
$\gamma=e_1\wedge e_4 - e_2\wedge e_3$.
Then $(\alpha,\alpha)=(\gamma,\gamma)=-2$,  $(\beta,\beta)=2$, and
$\psi=R_\alpha R_\beta R_\gamma$. In particular, $\det(\psi)=-1$ and $\psi$ does not belong
to $SO(\Wedge{2} V)$. $R_\alpha$ and $R_\gamma$ both change the orientation of the 
negative cone and preserve the orientation of the positive cone. $R_\beta$ changes the orientation of the 
positive cone and preserves the orientation of the negative cone. Hence, $\psi$ has the same affect as $R_\beta$
on the orientations of both cones.
\end{proof}

We need next the following analogue of Lemma 
\ref{lemma-signed-isometry-mapping-delta-to-delta-is-a-parallel-transport}.
Let $S$ be a $2$-dimensional compact complex torus and $\delta$ the class in 
$H^2(K^{[n]}(S),\Integers)$, which is half the class of
the diagonal divisor in $K^{[n]}(S)$.
Then $\delta^\perp$ is naturally isometric to $H^2(S,\Integers)$ \cite{beauville,yoshioka-abelian-surface}.
Now $H^2(S,\Integers)$ is isometric to $\Wedge{2} H^1(S,\Integers)$, endowed with 
the bilinear pairing $(x\wedge y,z\wedge w)=\int_Sx\wedge y\wedge z\wedge w$.
Set $Spin(S):=SL[H^1(S,\Integers)]$. We get  the natural surjective homomorphism
\begin{equation}
\label{eq-spin-to-orthogonal}
\Wedge{2} \ \ : \ Spin(S) \ \ \rightarrow \ \ SO^+\left[H^2(S,\Integers)\right].
\end{equation}
Let $S^*$ be the dual complex torus. Then $H^1(S^*,\Integers)$ is isomorphic to $H^1(S,\Integers)^*$.
We get a natural isomorphism $H^2(S,\Integers)\cong H^2(S,\Integers)^*\cong H^2(S^*,\Integers)$,
where the first isomorphism is induced by the intersection pairing.
Denote by $\bar{\tau}:H^2(S,\Integers)\rightarrow H^2(S^*,\Integers)$ the composite isomorphism above and let
\[
\tau:H^2(K^{[n]}(S),\Integers)\rightarrow H^2(K^{[n]}(S^*),\Integers)
\] 
be the isomorphism restricting to $\delta^\perp$ as $-\bar{\tau}$ and mapping the class $\delta$ of the
diagonal divisor  to the corresponding class in $H^2(K^{[n]}(S^*),\Integers)$.

\begin{prop}
\label{prop-when-is-an-isometry-of-generalized-kummers-a-parallel-transport}
Let $S_1$ and $S_2$ be $2$-dimensional compact complex tori.
Denote by $\delta_i$ the class in $H^2(K^{[n]}(S_i),\Integers)$, which is half the class of
the diagonal divisor in $K^{[n]}(S_i)$.
Let $g: H^2(K^{[n]}(S_1),\Integers)\rightarrow H^2(K^{[n]}(S_2),\Integers)$ be an isometry 
compatible with the orientations of the positive cones and satisfying 
$g(\delta_1)=\delta_2$. Let $\bar{g}:\delta_1^\perp\rightarrow \delta_2^\perp$
be its restriction. Denote by $b_i: \ \Wedge{2}H^1(S_i,\Integers)\rightarrow \delta_i^\perp$
the natural isometry.
\begin{enumerate}
\item
\label{lemma-item-g-lifts}
If the isometry $b_2^{-1}\bar{g}b_1: \ \Wedge{2}H^1(S_1,\Integers)\rightarrow \ \Wedge{2}H^1(S_2,\Integers)$
lifts to an oriented isomorphism 
$
\tilde{g}  :  H^1(S_1,\Integers) \rightarrow H^1(S_2,\Integers),
$
then $g$ is a parallel-transport operator.
\item
\label{lemma-item-otherwise}
Otherwise, $\tau\circ g:H^2(K^{[n]}(S_1),\Integers)\rightarrow H^2(K^{[n]}(S_2^*),\Integers)$ 
is a parallel-transport operator.
\end{enumerate}
\end{prop}

\begin{proof}
(\ref{lemma-item-g-lifts})
The proof is similar to that of Lemma 
\ref{lemma-signed-isometry-mapping-delta-to-delta-is-a-parallel-transport}.
Simply replace the reference to Borcea's result in \cite{borcea}, by the easier statement 
that any orientation preserving isomorphism from $H^1(S_1,\Integers)$ to $H^1(S_2,\Integers)$
is a parallel-transport operator.

(\ref{lemma-item-otherwise})
The isometry $b_2^{-1}\bar{g}b_1$ preserves the orientations of the positive cones, since $g$ does. 
The isometry $b_2^{-1}\bar{g}b_1$ reverses the orientations of the negative cones, since it does not lift to
an oriented isomorphism from $H^1(S_1,\Integers)$ to $H^1(S_2,\Integers)$. 
The isometry $-\bar{\tau}$ preserves the orientation of the positive cones and reverses the orientations of
the negative cones, by Lemma \ref{lemma-phi-does-not-lift}. Hence, the composition $-\bar{\tau}b_2^{-1}\bar{g}b_1$
preserves the orientations of both positive and negative cones, and hence lifts to an
orientated isomorphism from $H^1(S_1,\Integers)$ to $H^1(S_2^*,\Integers)$. 
The isometry $\tau g$ is thus a parallel-transport operator, by part \ref{lemma-item-g-lifts}.
\end{proof}

\begin{rem}
\label{rem-tau-not-parallel}
Note that the Hodge-isometry $\tau:H^2(K^{[n]}(S),\Integers) \rightarrow H^2(K^{[n]}(S^*),\Integers)$
preserves the orientation of the positive cones, but it is not a parallel transport operator. 
Indeed, if it were, then $K^{[n]}(S)$ and $K^{[n]}(S^*)$ would be bimeromorphic, 
by the Hodge-theoretic Torelli Theorem \cite[Theorem 1.3]{markman-survey}.
However, Namikawa observed that $K^{[n]}(S)$ and $K^{[n]}(S^*)$ are not bimeromorphic in general 
\cite{namikawa-torelli}.
\end{rem}

Let $X$ be an irreducible holomorphic symplectic manifold deformation equivalent to 
a $2n$-dimensional generalized Kummer variety. 
Recall that  $\W(X)$ is the  subgroup  of $O^+[H^2(X,\Integers)]$, consisting of
elements acting by $\pm 1$ on $[H^2(X,\Integers)]^*/[H^2(X,\Integers)]$, and
$\N(X)$ is the kernel of $\det\cdot \chi:\W(X)\rightarrow \{\pm 1\}$, by definition.
As a corollary of the above proposition and remark, we get the following statement.

\begin{cor}
\label{cor-monodromy-of-generalized-kummers-when-n+1-is-a-prime-power}
\begin{enumerate}
\item
\label{cor-item-intersection-is-N}
$Mon^2(X)\cap \W(X) = \N(X)$. 
\item
\label{cor-item-prime-power-case}
If $n+1$ is a prime power, then $Mon^2(X)=\N(X)$.
\end{enumerate}
\end{cor}

\begin{proof}
(\ref{cor-item-intersection-is-N})
$\N(X)$ is contained in $Mon^2(X)$, by Theorem \ref{thm-monodromy-of-kummers}. 
Hence, it suffices to find an element of $\W(X)$, which does not belong to 
$Mon^2(X)$. 
We may assume that $X=K^{[n]}(S)$, for a two-dimensional compact complex torus $S$.
Choose an orientation preserving  isomorphism $\tilde{h}:H^1(S^*,\Integers)\rightarrow H^1(S,\Integers)$.
Extend $\Wedge{2}\tilde{h}$ to an isometry $h:H^2(K^{[n]}(S^*),\Integers)\rightarrow H^2(K^{[n]}(S),\Integers)$
by sending the class of the diagonal divisor in $K^{[n]}(S^*)$ to that in $K^{[n]}(S)$.
Then $h\circ\tau$ belongs to $\W(X)$ but not to $\N(X)$. Now $\tau$ is not a parallel-transport
operator, by Remark \ref{rem-tau-not-parallel},
while $h$ is, by Proposition \ref{prop-when-is-an-isometry-of-generalized-kummers-a-parallel-transport}.
Hence, $h\circ\tau$ does not belong to $Mon^2(K^{[n]}(S))$.

(\ref{cor-item-prime-power-case}) 
The quotient $\Lambda^*/\Lambda$ is a cyclic group of order $2n+2$.
$O^+(\Lambda)$ surjects onto the subgroup of $\Aut[\Lambda^*/\Lambda]$
acting by multiplicative units $u$ in $\Integers/(2n+2)\Integers$, satisfying $u^2=1$,
by \cite[Theorem 1.14.2]{nikulin}.
If $n+1$ is a prime power, then  such a unit is $1$ or $-1$ \cite{oguiso}.
In that case $\W(X)=O^+[H^2(X,\Integers)]$ and 
the equality $Mon^2(X)=\N(X)$ follows from part \ref{cor-item-intersection-is-N}.
\end{proof}

We conjecture that the equality in part (\ref{cor-item-prime-power-case})  of the Corollary
holds, for all $n\geq 2$. Compare with \cite[Theorem 1.2]{markman-constraints}.

Let $S_0$ be a two-dimensional compact complex torus and $(K^{[n]}(S_0),\eta_0)$ a marked pair in 
$\fM^0_\Lambda$. Let $\delta_0$ be half the class of the diagonal divisor in $K^{[n]}(S_0)$.
Let $\Sigma'\subset \Sigma$ be the $\N$-orbit of $\eta_0(\delta_0)$.

\begin{lem}
\label{lemma-characterization-of-periods-of-Kummers}
Let $(X,\eta)$ be a marked pair in $\fM^0_\Lambda$, such that its period belongs to $\Omega_{\Sigma'}$.
Then $X$ is bimeromorphic to the generalized Kummer variety $K^{[n]}(S)$ of some 
two-dimensional compact complex torus $S$.
\end{lem}

\begin{proof} 
The proof is a translation of that of Lemma \ref{lemma-X-is-bimeromorphic-to-a-Hilbert-scheme}
via the following changes. Replace the group $W$ by the group $\N$. Replace 
Lemma \ref{lemma-signed-isometry-mapping-delta-to-delta-is-a-parallel-transport} 
by Proposition \ref{prop-when-is-an-isometry-of-generalized-kummers-a-parallel-transport}.
Replace Theorem \ref{thm-monodromy}
by Theorem \ref{thm-monodromy-of-kummers}.
The latter states only the inclusion of $\N(X)$ in $Mon^2(X)$, but only that inclusion is needed in the proof.
\end{proof}

\begin{proof} (Of Theorem \ref{thm-density-of-kummers}).
The proof is a translation of that of 
Theorem \ref{thm-density-in-moduli}, replacing Lemma \ref{lemma-X-is-bimeromorphic-to-a-Hilbert-scheme}
by Lemma \ref{lemma-characterization-of-periods-of-Kummers}
and Theorem \ref{thm-monodromy}
by Theorem \ref{thm-monodromy-of-kummers}.
\end{proof}


\end{document}